\documentclass{article}%
\usepackage{amsmath}
\usepackage{amsfonts}
\usepackage{amssymb}
\usepackage{graphicx}%
\providecommand{\U}[1]{\protect\rule{.1in}{.1in}}
\newtheorem{theorem}{Theorem}[section]

\newtheorem{lemma}[theorem]{Lemma}
\newtheorem{proposition}[theorem]{Proposition}
\newenvironment{proof}[1][Proof]{\noindent\textbf{#1.} }{\ \rule{0.5em}{0.5em}}
\begin{document}

\title{On some relatively free pro-$p$ groups}
\author{Dan Segal}
\maketitle

\section{Introduction}

A group $G$ is \emph{finitely axiomatizable} (or \emph{FA}) in some class
$\mathcal{C}$ if $G$ is determined up to isomorphism among all groups in
$\mathcal{C}$ by a single sentence in the first-order theory of groups (or,
equivalently, by finitely many such sentences). When $\mathcal{C}$ is a class
of profinite groups, \emph{isomorphism} is taken to mean topological
isomorphism (but the language of groups is the usual one, with no topological
features). See \cite{NST} for some background. That paper raises the question:
\emph{is a finitely generated free pro}-$p$ \emph{group FA in the class of
pro-}$p$\emph{ groups, or in the class of all profinite groups}? This is still
open; a small step in that direction is achieved in \cite{M}, which considers
relatively free metabelian pro-$p$ groups. Here we take another small step with

\begin{theorem}
\label{mainthm}The free $2$-generator centre-by-metabelian pro-$p$ group and
the free $2$-generator $\mathfrak{N}_{2}$-by-abelian pro-$p$ group are each FA
among all profinite groups.
\end{theorem}

($\mathfrak{N}_{2}$ means `nilpotent of class $2$'). It is conceivable that
similar methods would extend this result to all finitely generated
(non-procyclic) free $\mathfrak{N}_{c}$-by-abelian pro-$p$ groups with
$c\geq2$, but even this seems quite difficult; and it would probably not
suffice to answer the original question about absolutely free pro-$p$ groups.

\subsection{Group-theoretic results\label{gtr}}

We fix some notation and make some preliminary remarks.

$F$ denotes the free pro-$p$ group on generators $x,y$. The derived group of
$F$ is%
\[
D=F^{\prime}=[F,x][F,y];
\]
this is a free pro-$p$ group with the subspace topology from $F$ (Lemma
\ref{xy} below, and \cite{RZ}, Cor. 7.7.5).

By abuse of notation, I shall use $x,~y$ also to denote their images in
whatever ambient group they occur - what this means should be clear from the context.

$D_{2}=\overline{D^{\prime}}$ is the closure of the derived group of $D$, and
$F/D_{2}$ is the free metabelian pro-$p$ group on $x$ and $y$.
\[
Q=[D_{2},F]=[D_{2},x][D_{2},y]
\]
is closed (Lemma \ref{xy} below), and $F/Q$ is the free centre-by-metabelian
pro-$p$ group on $x$ and $y$.

$D_{3}=\overline{\gamma_{3}(D)}=\overline{[D_{2},D]}<Q$, and $F/D_{3}$ is the
free $\mathfrak{N}_{2}$-by-abelian pro-$p$ group on $x$ and $y$. Note that
$D/D_{3}$ is a relatively free $\mathfrak{N}_{2}$ pro-$p$ group, which implies
that $D_{2}/D_{3}=\mathrm{Z}(D/D_{3})$.

$X=\overline{\left\langle x,y\right\rangle }$ is the free abelian pro-$p$
group on $x$ and $y$.
\[
R=\mathbb{Z}_{p}[[X]]
\]
is the completed group algebra of $X$. \ We identify $X$ with $F/D$ and
consider $M_{1}:=D/D_{2}$ and $M_{2}:=D_{2}/D_{3}$ as ~$R$-modules. Note that
for each $i$ we have%
\[
\lbrack M_{i},F]=M_{i}(X-1)=M_{i}(x-1)+M_{i}(y-1).
\]
Set%
\[
u=[y,x].
\]
Then $D$ is topologically generated by conjugates of $u$, so%
\begin{equation}
D_{2}=[D,\,u]\cdot\lbrack D_{2},F]=[D,\,u]\cdot Q. \label{genD_2modQ}%
\end{equation}

The next two propositions are established in later sections:

\begin{proposition}
\label{P1}\emph{(i) }$M_{1}$ is the free cyclic $R$-module on $uD_{2}$.

\emph{(ii) }$M_{2}$ is torsion-free as an $R$-module.

\emph{(iii)} For $i=1,~2$ we have%
\[
M_{i}(x-1)\cap M_{i}(y-1)=M_{i}(x-1)(y-1).
\]

\emph{(iv)}%
\[%
%TCIMACRO{\dbigcap \limits_{1\neq g\in X}}%
%BeginExpansion
{\displaystyle\bigcap\limits_{1\neq g\in X}}
%EndExpansion
M_{i}(g-1)=0.
\]

\emph{(v)} $M_{2}/[M_{2},F]$ is torsion-free as a $\mathbb{Z}_{p}$-module.
\end{proposition}

Part (i) implies (the presumably well-known fact) that $F/D_{2}$ has trivial
centre .

\begin{proposition}
\label{P2}Let $C/Q=\mathrm{C}_{D/Q}(uQ).$ Then

\emph{(i) }$C/Q=\mathrm{C}_{D/Q}(aQ)$ for every $a\in C\smallsetminus D_{2}$,

\emph{(ii) }there exists $N<D$ with $C\cap N=D_{2}$ such that%
\[
N/Q=\mathrm{C}_{D/Q}(bQ)\text{ for every }b\in N\smallsetminus D_{2}%
\]
and
\begin{equation}
\frac{C}{D_{2}}\times\frac{N}{D_{2}}\supseteq\frac{D^{2}D_{2}}{D_{2}}.
\label{CXN}%
\end{equation}

\end{proposition}

We often use the fact that f.g. profinite groups are \emph{Hopfian}
(\cite{RZ}, Prop. 2.5.2), and the following

\begin{lemma}
\label{xy}If $N$ is a closed normal subgroup of a pro-$p$ group $G=\overline
{\left\langle x,y\right\rangle }$ then
\begin{equation}
\lbrack N,G]=[N,x][N,y]. \label{comdef2}%
\end{equation}

\end{lemma}

\begin{proof}
Follows from \cite{W}, Ex. 1.2.3.
\end{proof}

\medskip

The first-order expressibility of various properties of pro-$p$ groups is
established in \cite{NST}, Sections 2 and 5; these will sometimes be used
without special mention.

\section{Proof of Theorem \ref{mainthm}}

We start by quoting

\begin{theorem}
\label{meta}\emph{(\cite{M}, Th. 1) }The group $F/D_{2}$ is FA in profinite groups.
\end{theorem}

Set $\widetilde{F}=F/Q$ and write $\symbol{126}$ for the quotient map. It
follows from Prop. \ref{P1} (i), (ii) that the centre of $\widetilde{F}$ is
$\widetilde{D_{2}}$.

I claim that a pro-$p$ group $G$ is isomorphic to $\widetilde{F}$ iff it
satisfies the following axioms; that each of these can be expressed as a
first-order sentence follows from results in \cite{NST}, Section 5.1, together
with Lemma \ref{xy} and Theorem \ref{meta}.

\begin{enumerate}
\item[A1] $\mathrm{d}(G)=2.$

\item[A2] $G/\mathrm{Z}(G)\cong F/D_{2}.$

\item[A3] Set $Z=\mathrm{Z}(G)$. Then $Z^{2}\neq1.$

\item[A4] Fix a generating pair $\{x,y\}$, put $u=[y,x]$ and $C=\mathrm{C}%
_{G^{\prime}}(u)$. Then
\[
a\in C\smallsetminus Z\Longrightarrow\mathrm{C}_{G^{\prime}}(a)=C.
\]

\end{enumerate}

It follows from Propositions \ref{P1} and \ref{P2} that the axioms hold for
$\widetilde{F}$. Now suppose that the pro-$p$ group $G$ satisfies the axioms.
There is an epimorphism $\theta$ from $F$ to $G$ sending $x$ to $x$ and $y$ to
$y$ (abusing notation in the obvious way). Set $K=\ker\theta$. Axiom A2
implies that $K\leq D_{2}$; indeed, since $G/Z$ is metabelian, $D_{2}%
\theta\subseteq Z,$ while $G/D_{2}\theta\cong F/KD_{2}$, so $F/KD_{2}$ maps
onto $G/Z\cong F/D_{2}$. As $F/D_{2}$ is Hopfian this implies that
$KD_{2}=D_{2}$.

Also $K\geq\lbrack D_{2},F]=Q$ as $[D_{2},F]\theta\subseteq\lbrack Z,G]=1$,
and the centre of $F/K$ is exactly $D_{2}/K$ since $F/D_{2}$ has trivial
centre; so $D_{2}\theta=Z$.

Set $C(K)/K:=\mathrm{C}_{D/K}(uK)$. Thus $C(K)\geq C=\mathrm{C}_{D}%
(uQ\operatorname{mod}Q),$ and
\begin{equation}
K=[C(K),u]\cdot Q \label{[C(K),u]}%
\end{equation}
by (\ref{genD_2modQ}). Axiom A4 applied to $F/K\cong G$ gives%
\begin{equation}
a\in C(K)\smallsetminus D_{2}\Longrightarrow\mathrm{C}_{D/K}(aK)=C(K)/K.
\label{ax4applied}%
\end{equation}

Suppose that $K\neq Q$. Then $K^{2}\nsubseteq Q$ since $D_{2}/Q$ is
torsion-free by Proposition \ref{P1} (iv), whence $C(K)^{2}\nsubseteq C$. Let
$N$ be the subgroup given in Proposition \ref{P2}. It follows from (\ref{CXN})
that there exist $a\in N\smallsetminus D_{2}$ and $c\in C$ such that $ca\in
C(K)^{2}$, and then $a\in N\cap C(K)\smallsetminus D_{2}$.

Now $[a,N]\subseteq N^{\prime}\subseteq Q\subseteq K$ with (\ref{ax4applied})
gives $N\leq$ $C(K)$. As also $C\leq C(K)$ it follows that $C(K)\geq D^{2}. $
Hence
\[
K\geq\lbrack D^{2},u]Q\supseteq D_{2}^{2},
\]
by (\ref{[C(K),u]}) and (\ref{genD_2modQ}). Thus%
\[
Z^{2}=D_{2}^{2}\theta=1,
\]
contradicting Axiom A3. It follows that $K=Q,$ and $G\cong
F/K=F/Q=\widetilde{F}$ as claimed.

We have shown that $\widetilde{F}$ is FA among pro-$p$ groups. \ Let $\sigma$
be a sentence that characterizes this group among pro-$p$ groups. Let $\tau$
be a sentence asserting for a profinite group $G$ that $G/\mathrm{Z}(G)\cong
F/D_{2}$ (Theorem \ref{meta}). Let $\gamma$ be a sentence asserting that for
some elements $x,~y$ we have%
\[
\mathrm{Z}(G)\subseteq\lbrack G,x][G,y].
\]
I claim that $\tau\wedge\sigma\wedge\gamma$ determines $\widetilde{F}$ among
all profinite groups.

Indeed, suppose $G$ is a profinite group such that $G/\mathrm{Z}(G)\cong
F/D_{2}$. Then all Sylow pro-$q$ subgroups for $q\neq p$ are contained in
$\mathrm{Z}(G),$ so $G=Q\times P$ for a pro-$p$ group $G$ and an abelian
pro-$p^{\prime}$ group $Q$. Then $Q\subseteq\mathrm{Z}(G)\subseteq G^{\prime
}\subseteq P,$ so $Q=1$ and $G$ is a pro-$p$ group. Now $\sigma$ implies that
$G\cong\widetilde{F}$.

The first claim of Theorem \ref{mainthm} is now proved.

\medskip

Next, set $\widetilde{F}=F/D_{3}$, and write $x,$ $y$ for their images in
$\widetilde{F}$. Recall that $M_{2}=D_{2}/D_{3}$ which is the centre of
$D/D_{3}=$ $\widetilde{F}^{\prime}=[\widetilde{F},x][\widetilde{F},y]$.
Writing $M$ for $M_{2}$ and using additive notation we have%
\[
\lbrack M,\widetilde{F}]=[M,x]+[M,y]
\]
and
\[
\lbrack M,x]\cap\lbrack M,y]=[M,x,y]
\]
by Proposition \ref{P1} (iii). Also each map $M\rightarrow M;~a\mapsto\lbrack
a,g]$ for $1\neq g\in X,$ is injective, by Proposition \ref{P1} (ii).

Let $K$ be a closed normal subgroup of $F$ with $[D_{2},F]\geq K\geq D_{3}$.
Thus%
\[
\widetilde{K}:=K/D_{3}\leq\lbrack M,\widetilde{F}]=[M,x]+[M,y].
\]
Write $\pi:\widetilde{F}\rightarrow\widetilde{F}/\widetilde{K}$ for the
quotient map, and assume

\begin{enumerate}
\item[P1] Each of the maps $M\pi\rightarrow M\pi;~a\mapsto\lbrack a,x],$
$a\mapsto\lbrack a,y]$, is injective,

\item[P2]
\[
\lbrack M,x]\pi\cap\lbrack M,y]\pi=[M,x,y]\pi.
\]

\end{enumerate}

Now suppose $w=[a,x]+[b,y]\in\widetilde{K}$ with $a,~b\in M$. Then
$[a\pi,x]=[c\pi,x,y]=[c\pi,y,x]$ for some $c\in M,$ whence $a\pi=[c\pi,y].$
Then
\[
\lbrack b\pi,y]=-[a\pi,x]=[-c\pi,y,x]=[-c\pi,x,y],
\]
whence $b\pi=[-c\pi,x]$. Thus there exist $\alpha,~\beta\in\widetilde{K}$ such
that ~$a=[c,y]+\alpha,~b=-[c,x]+\beta$, and we obtain%
\begin{align*}
w  &  =[c,y,x]+[\alpha,x]-[c,x,y]+[\beta,y]\\
&  =[\alpha,x]+[\beta,y]\in\lbrack\widetilde{K},\widetilde{F}].
\end{align*}
Thus $\widetilde{K}=[\widetilde{K},\widetilde{F}],$ and as $\widetilde{F}$ is
a pro-$p$ group it follows that $\widetilde{K}=1$.

\medskip

It is now easy to establish the \emph{claim}: a pro-$p$ group $G$ is
isomorphic to $\widetilde{F}$ iff it satisfies the following axioms:

\begin{enumerate}
\item[B1] $\mathrm{d}(G)=2.$

Pick generators $x,~y$ for $G,$ set $D^{\ast}:=[G,x][G,y]$ and $M=\mathrm{Z}%
(D^{\ast}).$

\item[B2] $[D^{\ast},D^{\ast}]\subseteq M.$

\item[B3] $G/[M,G]\cong F/[D_{2},F].$

\item[B4] Each of the maps $M\rightarrow M;~a\mapsto\lbrack a,x],$
$a\mapsto\lbrack a,y],$ is injective.

\item[B5] $[M,x]\cap\lbrack M,y]=[M,x,y]$.
\end{enumerate}

Since $G$ is a pro-$p$ group, $D^{\ast}$ is equal to $G^{\prime}$. That B3 is
expressible by a first-order sentence is what we have proved above.

Suppose that $G$ is a such a group. As above, there is an epimorphism $\theta$
from $F$ to $G$ sending $x$ to $x$ and $y$ to $y$ (abusing notation in the
obvious way). Then $D\theta=D^{\ast}$ and $D_{3}\theta=\overline{[D^{\ast
},D^{\ast},D^{\ast}]}\subseteq\overline{[M,D^{\ast}]}=\{1\}$, so
$K:=\ker\theta\geq D_{3}$.

\bigskip Axioms B3 and B2 provide epimorphisms%
\[
\frac{F}{[D_{2},F]}\overset{\theta^{\ast}}{\rightarrow}\frac{G}{[\overline
{D^{\ast\prime}},G]}\rightarrow\frac{G}{[M,G]}\overset{\cong}{\rightarrow
}\frac{F}{[D_{2},F]}.
\]
The Hopf property of $F/[D_{2},F]$ implies that these are all isomorphisms; it
follows that $K\leq\lbrack D_{2},F]$.

Write $\pi:\widetilde{F}\rightarrow\widetilde{F}/\widetilde{K}$ for the
quotient map, as above. Then Axioms B4 and B5 are equivalent to the conditions
P1 and P2. We have shown that these imply $\widetilde{K}=1$. Thus $\theta$
induces an isomorphism from $\widetilde{F}$ to $G$, and the claim is established.

To complete the proof of Theorem \ref{mainthm} it remains to show that
$\widetilde{F}$ is FA among all profinite groups. This depends on\bigskip

\begin{lemma}
\label{prof} \emph{(\cite{S}, Lemma 4.6) }Let $G$ be profinite group and $M$ a
profinite $G$-module. Put $C=\mathrm{C}_{G}(M)$. Suppose that%
\begin{align*}
pM+M(G-1)  &  \neq M,\\%
%TCIMACRO{\dbigcap \limits_{g\in G\smallsetminus C}}%
%BeginExpansion
{\displaystyle\bigcap\limits_{g\in G\smallsetminus C}}
%EndExpansion
M(g-1)  &  =0,\\
0\neq a\in M  &  \rightarrow\mathrm{C}_{G}(a)=C.
\end{align*}
Then $M$ is a pro-$p$ group.
\end{lemma}

Note that these hypotheses are satisfied with $M=M_{2}$ and $G=\widetilde{F};
$ the first one is trivial since $\widetilde{F}$ is a pro-$p$ group, the
others follow from Proposition \ref{P1} (iv), (ii).

Let $\sigma$ be a sentence that characterizes $\widetilde{F}$ among pro-$p$
groups. I claim that a profinite group $G$ is isomorphic to $\widetilde{F}$
iff it satisfies the following axioms in addition to $\sigma$:

\begin{enumerate}
\item[C1] There exist $x,~y$ such that the set $[G,x][G,y]$ contains every
product of $3$ commutators.

Pick such a pair $x,~y$ and set
\begin{align*}
D^{\ast}  &  :=[G,x][G,y],\\
M  &  =\mathrm{Z}(D^{\ast}).
\end{align*}
(It follows that $D^{\ast}=G^{\prime}$ and is closed).

\item[C2] The hypotheses of Lemma \ref{prof}.

\item[C3] $G/M\cong F/D_{2}$.
\end{enumerate}

As before, Axiom C3 is legitimate in view of Theorem \ref{meta}. If $G$
satisfies these axioms then $M$ is a pro-$p$ group by Lemma \ref{prof}, and
$G/M$ is is a pro-$p$ group in view of C3, so $G$ is a pro-$p$ group, and then
$\sigma$ implies that $G\cong\widetilde{F}$.

This completes the proof of Theorem \ref{mainthm}.

\section{The setup}

The rest of the paper is mostly devoted to Propositions \ref{P1} and \ref{P2}.
The first of these actually holds in greater generality: $M_{2}$ can be
replaced by $M_{n}=\overline{\gamma_{n}(D)}/\overline{\gamma_{n+1}(D)}$ for
any $n\geq2$ (and the proof is in principle the same, just notationally more
complicated). The difficulty in generalizing Theorem \ref{mainthm} lies in the
lack, so far, of a substitute for Proposition \ref{P2}.

The short final section is by way of an aside: we show that the ring $R$ can
be interpreted in $F$ via Proposition \ref{P1}(i).

\subsection{Structure of $M$\label{Mnstruct}}

Keep the notation as before. $X$ is the free abelian pro-$p$ group on
$\{x,y\}$. According to \cite{DDMS} Thm. 7.20 the completed group algebra
$R:=\mathbb{Z}_{p}[[X]]$ is equal to the power series algebra $\mathbb{Z}%
_{p}[[\xi,\eta]]$ where $\xi=x-1,$ $\eta=y-1$.

\bigskip To simplify notation, let us assume in this section that $D_{3}=1,$
and write%
\[
\symbol{126}:F\rightarrow F/D_{2}%
\]
for the quotient map.

Proposition \ref{P1}(i) claims that $\widetilde{D}$ is the free cyclic
$R$-module on $\widetilde{u}.$ This is the special case $d=2$ of \cite{M},
Prop. 8. It allows us to identify $\widetilde{D}$ with $R$ via $\widetilde{u}%
r\mapsto r$.

Set $M=M_{2}=D_{2}$. Then $M$ is an $R$-module, via%
\[
a\cdot g=a^{g}~~(g\in X).
\]
For $w_{1},~w_{2}\in D$ we set%
\[
(\widetilde{w_{1}},\widetilde{w_{2}}):=[w_{1},w_{2}]
\]
This gives a bilinear map from $\widetilde{D}$ into $M,$ whose image generates
$M$ as an abelian pro-$p$ group.

Set $A=\widetilde{D}.$ According to \cite{RZ}, Ex. 5.5.5,%
\begin{align*}
T:=A\widehat{\otimes\!}A\cong &  \mathbb{Z}_{p}[[X_{1}\times X_{2}]]:=S\\
g\otimes h\longmapsto &  g(1)h(2)~~(g,~h\in X)
\end{align*}
where $g\longmapsto g(i)$ is an isomorphism from $X$ to $X_{i}=\overline
{\left\langle x(i),y(i)\right\rangle }$. The induced isomorphism from $R$ to
$R_{i}:=\mathbb{Z}_{p}[[X_{i}]]\subseteq S$ is written $r\longmapsto r(i)$.
For simplicity I will write $\xi_{i}=\xi(i),~~\eta_{i}=\eta(i)$.

Thus $T$ is identified with $S$ via%
\begin{equation}
r\otimes s\mapsto r(1)s(2). \label{tensor-product}%
\end{equation}

Recall ( \cite{DDMS} Thm. 7.20) that%
\[
S=\mathbb{Z}_{p}[[\xi_{1},\eta_{1},\xi_{2},\eta_{2}]].
\]
Let $X$ act componentwise on the tensors; then $g\in X$ acts on $S$ as
multiplication by
\[
\widehat{g}:=g(1)g(2).
\]

Now $g\mapsto\widehat{g}$ induces an embedding $r\mapsto\widehat{r}$ of $R$
into $S$, with $\widehat{R}=\mathbb{Z}_{p}[[\widehat{X}]]=\mathbb{Z}%
_{p}[[\zeta,\tau]]$ where
\[
\zeta=\widehat{\xi},~~\tau=\widehat{\eta}.
\]

Replacing $x(2)$ by $z=\widehat{x}$ and $y(2)$ by $t=\widehat{y}$ we see that
$X_{1}\times X_{2}$ has a basis
\[
\{x(1),y(1),z,t\}.
\]
So we also have%
\[
S=\mathbb{Z}_{p}[[\xi_{1},\eta_{1},\zeta,\tau]].
\]

\begin{lemma}
\label{Sasmodule}\emph{(i)} $S$ is torsion-free as an $R$-module.

\emph{(ii)}%
\begin{equation}
S\zeta\cap S\tau=S\zeta\tau. \label{intersection}%
\end{equation}

\emph{(iii)}
\begin{equation}%
%TCIMACRO{\dbigcap \limits_{1\neq g\in X}}%
%BeginExpansion
{\displaystyle\bigcap\limits_{1\neq g\in X}}
%EndExpansion
S(\widehat{g}-1)=0. \label{inters}%
\end{equation}

\end{lemma}

\begin{proof}
Claim (ii) is clear, and (i) is clear since $R$ acts as multiplication by
$\widehat{R}\subseteq S$ and $S$ is a domain. For (iii), let $W$ denote the
intersection in (\ref{inters}). Note now that if $g$ is a basis element of the
free abelian pro-$p$ group $X,$ say $X=\overline{\left\langle g,g^{\prime
}\right\rangle }$, then $S=\mathbb{Z}_{p}[[\xi_{1},\eta_{1},\gamma
,\gamma^{\prime}]]$ where $\gamma=(\widehat{g}-1),$ $\gamma^{\prime
}=(\widehat{g^{\prime}}-1),$ so $S\gamma$ is a prime ideal. If $W\neq0$ then
$S\gamma$ is a minimal prime of $W$. But the number of such minimal primes is
finite, while the number of distinct ideals $S\gamma$ is infinite (let $g$
range over the elements $xy^{j},$ $j\in\mathbb{N},$ for example).
\end{proof}

\begin{proposition}
\label{KZprop}There is an injective morphism $\breve{\Theta}$ of abelian
pro-$p$ groups $M\rightarrow T$ such that%
\[
(a,b)\breve{\Theta}=a\otimes b-b\otimes a
\]
for all $a,~b\in A$.
\end{proposition}

\begin{proof}
This follows from the embedding of the free Lie algebra in the corresponding
free associative algebra, as explained in the proof of \cite{KZ}, Cor. 2.3;
here we specialize to the case $n=2$.\emph{ }
\end{proof}

\medskip

Now set $\gamma=(12)\in\mathrm{Sym}_{2}$, and let $\mathrm{Sym}_{2}$ act on
$S$ by permuting $\{g(1),g(2)\}$ for all $g\in X$, so $(r(1)s(2))\gamma
=r(2)s(1)$ for $r,~s\in R$. Using our identifications of $A$ with $R$ and of
$T$ with $S$, we see that $\breve{\Theta}$ becomes an embedding%
\[
\Theta:M\rightarrow S,
\]
where now for $r,s\in R,$%
\begin{align}
(\widetilde{u}r,\widetilde{u}s)\Theta &  =r(1)s(2)-s(1)r(2)\nonumber\\
&  =\left(  r(1)s(2)\right)  (1-\gamma)\text{.} \label{sigmasum}%
\end{align}
It follows that
\begin{equation}
M\Theta=S(1-\gamma). \label{MthetaSzeta}%
\end{equation}

Since $\widehat{g}^{\gamma}=\widehat{g}$ for $g\in X$ , we have
\begin{align*}
(\widetilde{u}r,\widetilde{u}s)g\Theta &  =\left(  r(1)g(1)s(2)g(2)\right)
(1-\gamma)\\
&  =\left(  r(1)s(2)\widehat{g}\right)  (1-\gamma)\\
&  =\left(  r(1)s(2)\right)  (1-\gamma)\widehat{g}\\
&  =(\widetilde{u}r,\widetilde{u}s)\Theta\widehat{g}.
\end{align*}

It follows that $bg\Theta=b\Theta\widehat{g}$ for $b\in M$ and $g\in X$.

Set%
\[
\Delta=S\zeta+S\tau.
\]

The following two identities will be established in Subsection \ref{finalsec}:%
\begin{align}
S(1-\gamma)\cap S\zeta &  =(S(1-\gamma))\zeta=(S\zeta)(1-\gamma),
\label{Szeta}\\
S(1-\gamma)\cap\Delta &  =\Delta(1-\gamma). \label{MDelta}%
\end{align}

With (\ref{MthetaSzeta}) these amount to

\begin{lemma}
\label{keylem}%
\begin{align*}
M\cap S\zeta &  =M(x-1),\\
M\cap\Delta &  =M(X-1).
\end{align*}

\end{lemma}

We can now establish Proposition \ref{P1}(iii) (this uses only the first claim
of the lemma):

\begin{proposition}
\label{MxiintersectMeta}If $a,~b\in M$ and $a(x-1)=b(y-1)$ then $a=c(y-1)$ and
$b=c(x-1)$ for some $c\in M$. Thus%
\[
M(x-1)\cap M(y-1)=M(x-1)(y-1).
\]

\end{proposition}

\begin{proof}
Suppose $d=a(x-1)=b(y-1)$. Then%
\begin{align*}
b\Theta\tau &  =b(y-1)\Theta\\
&  =a(x-1)\Theta=a\Theta\zeta=s\zeta\tau
\end{align*}
for some $s\in S,$ by (\ref{intersection}). As multiplication by $\tau$ is
injective, Lemma \ref{keylem} shows that%
\[
b\Theta=s\zeta=c(x-1)\Theta
\]
for some $c\in M$. It follows that $b=c(x-1)$, and that%
\[
a\Theta\zeta=c(x-1)\Theta\tau=c\Theta\tau\zeta=c(y-1)\Theta\zeta
\]
whence $a=c(y-1)$. Finally,%
\[
d=b(y-1)=c(x-1)(y-1)\in M(x-1)(y-1).
\]

\end{proof}

\medskip

Now $\Delta$ is the kernel of the augmentation mapping $S\rightarrow R(1)$
that sends $\widehat{X}$ to $\{1\}.$ Let $\pi:S\rightarrow R$ be the
composition of this map with the isomorphism $R(1)\rightarrow R,~r(1)\mapsto
r.$ Thus for $g\in X$ we have%
\begin{align}
g(1)\pi &  =g\label{pi}\\
g(2)\pi &  =g(1)^{-1}\widehat{g}\pi=g^{-1}.\nonumber
\end{align}

According to Lemma \ref{keylem},%
\begin{equation}
\lbrack M,F]=M\cap\Delta=M\cap\ker\pi\text{.} \label{ker_pi}%
\end{equation}
So we may identify $M/[M,F]$ with $S(1-\gamma)\pi\subseteq R$. In particular
this implies the claim of Proposition \ref{P1}(v):

\begin{itemize}
\item $M/[M,F]$ is torsion-free for $\mathbb{Z}_{p}$.
\end{itemize}

\bigskip\ Now denote by $\ast$ the automorphism of $R$ induced by the
inversion automorphism of $X$. It follows from (\ref{pi}) that for $r,$ $s\in
R$ we have%
\[
r(1)s(2)\pi=rs^{\ast}.
\]

Consider elements $\widetilde{u}r,~\widetilde{u}s\in\widetilde{D}$. The image
of $(\widetilde{u}r,\widetilde{u}s)$ in $M\Theta$ is $r(1)s(2)-s(1)r(2).$ This
maps under $\pi$ to $rs^{\ast}-sr^{\ast}$. Thus%

\begin{align}
(\widetilde{u}r,\widetilde{u}s)\pi &  =rs^{\ast}-sr^{\ast}\label{bracketpi}\\
&  =(\widetilde{u}rs^{\ast},\widetilde{u})\pi,\nonumber
\end{align}
whence%
\begin{equation}
(\widetilde{u}r,\widetilde{u}s)\equiv(\widetilde{u}rs^{\ast},\widetilde{u}%
)~(\operatorname{mod}[M,F]). \label{multmodQ}%
\end{equation}

Let
\begin{align*}
R_{+}  &  =\{r\in R\mid r=r^{\ast}\},\\
R_{-}  &  =\{r\in R\mid r=-r^{\ast}\}.
\end{align*}
As $\ast$ is an involution we have $2R\subseteq R_{+}\oplus R_{-}.$

Fix $r\in R_{+}\smallsetminus\{0\}.$ Then for any $s\in R$ we have
\begin{align*}
(\widetilde{u}r,\widetilde{u}s)\pi &  =0\\
&  \Longleftrightarrow r(s^{\ast}-s)=0\\
&  \Longleftrightarrow s\in R_{+}.
\end{align*}

Fix $r\in R_{-}\smallsetminus\{0\}.$ Then for any $s\in R$ we have%
\begin{align*}
(\widetilde{u}r,\widetilde{u}s)\pi &  =0\\
&  \Longleftrightarrow r(s^{\ast}+s)=0\\
&  \Longleftrightarrow s\in R_{-}.
\end{align*}

Thus setting%
\begin{align*}
\widetilde{u}R_{+}  &  =C/D_{2}\\
\widetilde{u}R_{-}  &  =N/D_{2}%
\end{align*}
we have established the following (recall that $Q=[D_{2},F]=[M,F]$): \emph{the
group} $F/Q$ \emph{contains abelian subgroups} $C/Q=\mathrm{C}_{D/Q}(uQ)$
and$~N/Q$ \emph{such that}%
\begin{align*}
D^{2}  &  \subseteq CND_{2}\leq D\\
C\cap N  &  =D_{2},
\end{align*}%
\begin{align*}
a  &  \in C\smallsetminus D_{2}\Longrightarrow\mathrm{C}_{D/Q}(aQ)=C/Q,\\
b  &  \in N\smallsetminus D_{2}\Longrightarrow\mathrm{C}_{D/Q}(bQ)=N/Q.
\end{align*}
This is precisely the content of Proposition \ref{P2}.

\subsection{Group algebra notes\label{finalsec}}

It remains to establish the claims (\ref{Szeta}) and (\ref{MDelta}). To this
end we now forget all the preceding notation and begin anew. Let%

\[
X=\overline{\left\langle x_{1},x_{2}\right\rangle },~X(i)=\overline
{\left\langle x_{i}\right\rangle }%
\]
be free abelian pro-$p$ groups on the exhibited generators.

Set $\xi_{i}=x_{i}-1,~$ $z=x_{1}x_{2},~~\zeta=z-1$.

$T=\mathbb{Z}_{p}[[H]]$, $H$ a f.g. free abelian pro-$p$ group,
\begin{align*}
S  &  :=\mathbb{Z}_{p}[[H\times X]]=T[[X]]=T[[\xi_{1},\xi_{2}]]\\
S(i)  &  :=T[[X(i)]]=T[[\xi_{i}]].
\end{align*}
Then for $i=1,~2$ we have $X=X(i)\times\overline{\left\langle z\right\rangle
}=\overline{\left\langle z~,x_{i}\right\rangle }$ so%
\begin{align*}
S  &  =S(i)[[\zeta]]\\
&  =S\zeta\oplus S(i).
\end{align*}
Define $\pi_{i}:S\rightarrow S(i)\subseteq S$ by
\[
(a\zeta+c)\pi_{i}=c~~~(a\in S,~c\in S(i)).
\]
Now $r\mapsto r\zeta$ is an isomorphism from $S$ to $S\zeta$; let $\zeta^{-1}
$ denote the inverse isomorphism.

As $S(1-\pi_{i})=S\zeta$ we can define $\Psi_{i}:S\rightarrow S$ by%
\[
\Psi_{i}=(1-\pi_{i})\zeta^{-1};
\]
In particular,%
\[
\Psi_{i\left\vert S\zeta\right.  }=\zeta^{-1}.
\]
Now define%
\[
\Psi:S\rightarrow\frac{1}{2}S\subseteq\mathbb{Q}_{p}[[H\times X]]
\]
by%
\[
\Psi=\frac{1}{2}(\Psi_{1}+\Psi_{2}).
\]
Then%
\[
\Psi_{\left\vert S\zeta\right.  }=\zeta^{-1}:S\zeta\rightarrow S.
\]

Next, let $\mathrm{Sym}_{2}$ act on $X$ by permuting $x_{1},x_{2}$. Assume
also that $\mathrm{Sym}_{2}$ acts on $H,$ in some way. This extends to an
action on $S$, that fixes $\zeta$ and permutes the $S(i)$. Let $\gamma
=(12)\in\mathrm{Sym}_{2}$. Thus for%
\[
r=a_{i}\zeta+c_{i}~~(a_{i}\in S,~c_{i}\in S(i),~i=1,~2)
\]
we have%
\[
r^{\gamma}=a_{i}^{\gamma}\zeta+c_{i}^{\gamma}%
\]
so $r^{\gamma}\Psi_{i\gamma}=(r\Psi_{i})^{\gamma}$. Thus%
\begin{align*}
2(r\Psi)^{\gamma} &  =(r\Psi_{1}+r\Psi_{2})^{\gamma}\\
&  =r^{\gamma}\Psi_{2}+r^{\gamma}\Psi_{1}\\
&  =2r^{\gamma}\Psi;
\end{align*}
whence $\Psi$ is $\mathrm{Sym}_{2}$-\emph{equivariant}.

\medskip

Now (\ref{Szeta}) is exactly the statement of the following lemma,
taking$~H:=\overline{\left\langle y_{1},y_{2}\right\rangle }$:

\begin{lemma}
\label{SsandSzeta}Let $\mathbf{s}=1-\gamma\in\mathbb{Z}_{p}[\mathrm{Sym}_{2}%
]$. Then
\[
S\mathbf{s}\cap S\zeta=(S\mathbf{s})\zeta=(S\zeta)\mathbf{s.}%
\]

\end{lemma}

\begin{proof}
Say $a\mathbf{s}=b\zeta$ ($a,~b\in S$). Then%
\begin{align*}
b  &  =b\zeta\Psi\\
&  =a\Psi-a^{\gamma}\Psi\\
&  =(a\Psi)\mathbf{s}%
\end{align*}
so $b\zeta\in(S\mathbf{s})\zeta$. The rest follows because $c^{\gamma}%
\zeta=(c\zeta)^{\gamma}$ implies $(c\mathbf{s})\zeta=(c\zeta)\mathbf{s}$ for
$c\in S.$
\end{proof}

\medskip

Now let $Y=\overline{\left\langle y_{1},y_{2}\right\rangle }$ be another free
abelian pro-$p$ group on the exhibited generators, set $\eta_{i}=y_{i}-1,~$
$t=y_{1}y_{2},~~\tau=t-1$. Take $\mathrm{Sym}_{2}$ to act on $Y$ by permuting
$y_{1},y_{2}$. This induces an action on $H:=Y/\overline{\left\langle
t\right\rangle }$ (where $\gamma$ acts as inversion).

We have a natural epimorphism%
\[
R:=\mathbb{Z}_{p}[[Y\times X]]\rightarrow S=\mathbb{Z}_{p}[[H\times X]]
\]
with kernel $R\tau$. The preceding lemma now implies that%
\[
R\mathbf{s}\cap(R\tau+R\zeta)\subseteq R\tau+(R\mathbf{s})\zeta.
\]
Suppose that $r,~a,~b\in R$ and $r\mathbf{s}=a\tau+b\zeta$. Then
$r\mathbf{s}=a^{\prime}\tau+b^{\prime}\mathbf{s}\zeta$ for some $a^{\prime}$
and $b^{\prime}$, and then $(r-b^{\prime}\zeta)\mathbf{s}\in R\tau$.

Now Lemma \ref{SsandSzeta}, taking $X$ for $H$ and $Y$ for $X$, gives
$R\mathbf{s}\cap R\tau=R\mathbf{s}\tau$. So $(r-b^{\prime}\zeta)\mathbf{s}%
=c\mathbf{s}\tau$ for some $c.$ Then%
\[
r\mathbf{s}=b^{\prime}\zeta\mathbf{s}+c\mathbf{s}\tau\in(R\tau+R\zeta
)\mathbf{s}.
\]
Thus we have established

\begin{lemma}
\label{RsmeetDelta}Let $\mathbf{s}=1-\gamma\in\mathbb{Z}_{p}[\mathrm{Sym}%
_{2}]$. Then
\[
R\mathbf{s}\cap(R\tau+R\zeta)=(R\tau+R\zeta)\mathbf{s.}%
\]

\end{lemma}

This gives (\ref{MDelta}) (here $R$ is playing the role of $S$ in subsection
\ref{Mnstruct}).

\section{Appendix: Interpreting $R$}

The following material is not needed for the main results, but may be of some
use when we try to generalize these. It might also serve in alternative proof
of Theorem \ref{mainthm}, if we can strengthen it to a bi-interpretation
between the group $F/Q$ and the ring-with-involution $(R,\ast)$; this seems
plausible, but needs more work.

\begin{proposition}
The ring $R$ and its involution $\ast$ are interpretable (with parameters) in
the group $F/Q$.
\end{proposition}

\begin{proof}
From now on I assume that $Q=1,$ and write $F$ in place of $F/Q$ and similarly
for $D$ and $D_{2}=M$, so $[M,F]=1$. Modulo this convention we revert to the
notation of Sections \ref{gtr} and \ref{Mnstruct}. Now we identify $R$ with
\thinspace the definable section $D/D_{2}$ of $F$, via%
\[
r\longmapsto\widetilde{u}r
\]
(Proposition \ref{P1}(i)). In the framework of \cite{H}, Sec. 5.3, this
corresponds to a surjective map $f$ from the definable subgroup $D$ of $F$ to
$R$, such that for $v,w\in D$,%
\begin{align*}
\widetilde{u}(vf)  &  =vD_{2}\\
wf  &  =vf\Longleftrightarrow v^{-1}w\in D_{2},
\end{align*}
$D_{2}$ being definable as the centre of $F$.

Addition in $R$ corresponds to the group operation in $D/D_{2}$, in other
words%
\[
wf+vf=zf\Longleftrightarrow z^{-1}vw\in D_{2}.
\]

Recall that the involution \thinspace$\ast$ extends the inversion mapping on
$X$. At the end of Section \ref{Mnstruct} we defined subgroups%
\begin{align*}
C/D_{2}  &  =\{\widetilde{u}r\mid r=r^{\ast}\},\\
N/D_{2}  &  =\{\widetilde{u}r\mid r=-r^{\ast}\}.
\end{align*}
Now $u\in C$ and $u^{x}.u^{-x^{-1}}\in N$; applying the final statement of
Section \ref{Mnstruct} then shows that $C=\mathrm{C}_{D}(u)$ and
$N=\mathrm{C}_{D}(u^{x}.u^{-x^{-1}})$. Thus $C$ and $D$ are definable in $F$.
The relation $s=r^{\ast}$ in $R$ is now definable as follows:
\[
wf=(zf)^{\ast}\Longleftrightarrow\left(  wz\in C\right)  ~\wedge~(wz^{-1}\in
N).
\]

It remains to define the ring multiplication; for this it suffices to define
the relation $rs^{\ast}=t$. I claim that for $v,~w,~z\in D$, the following
statement is equivalent to $(vf).(wf)^{\ast}=zf$:%
\begin{equation}
\forall g~.[v,w^{g}]=[z,u^{g}]. \label{prodformula}%
\end{equation}

To see this, recall, (\ref{multmodQ}): this shows that for $r,~s\in R$ we
have
\[
(\widetilde{u}r,\widetilde{u}s)=(\widetilde{u}rs^{\ast},\widetilde{u}).
\]
Putting $r=vf$, $s=w^{g}f$ and $rs^{\ast}=z^{g^{-1}}f$ this gives
\[
\lbrack v,w^{g}]=[z^{g^{-1}},u]=[z,u^{g}]
\]
as required. Conversely, let $r=vf$, $s=wf$ and say $rs^{\ast}=tf$ . Now
(\ref{prodformula}) asserts%
\begin{align*}
\lbrack z,u^{g}]  &  = [v,w^{g}]\\
&  =(\widetilde{u}r,\widetilde{u}sg) =(\widetilde{u}rs^{\ast}g^{-1}%
,\widetilde{u})\\
&  =[t^{g^{-1}},u]=[t,u^{g}]
\end{align*}
for all $g\in G$ (in the middle line, interpret $g$ as $gD\in X$). As the
elements $u^{g}$ generate $D$ it follows that $t\equiv z~$modulo
$\mathrm{Z}(D).$ But $\mathrm{Z}(D)=D_{2}$ as is clear from Proposition
\ref{P2}. Thus $zf=tf=rs^{\ast}$ as claimed.
\end{proof}

\end{document}